\documentclass[12pt]{article}
\usepackage{latexsym}
\usepackage{amsmath}
\usepackage{amssymb}
\usepackage{amscd}
\usepackage{array}
\usepackage{comment}
\usepackage[all]{xy}
\usepackage{amsthm}
\usepackage{amsfonts}
\usepackage{enumerate}
\usepackage{hyperref}
\usepackage{mathtools}
\usepackage{stackrel}
\numberwithin{equation}{section}

\newtheorem{theorem}{Theorem}[]
\newtheorem{proposition}[theorem]{Proposition}
\newtheorem{lemma}[theorem]{Lemma}

\theoremstyle{definition}

\newtheorem{examples}[theorem]{Examples}

\newtheorem{remark}[theorem]{Remark}

\newcommand{\Z}{\mathbf Z}

\newcommand{\Hol}{\mathrm{Hol}}
\newcommand{\Norm}{\mathrm{Norm}}

\newcommand{\Ker}{\operatorname{Ker}}
\newcommand{\Hom}{\operatorname{Hom}}

\newcommand{\conj}{\mathrm{conj}}

\newcommand{\GL}{\mathrm{GL}}

\newcommand{\End}{\operatorname{End}}

\newcommand{\Aut}{\operatorname{Aut}}

\newcommand{\Id}{\operatorname{Id}}

\begin{document}

\begin{center}
\Large{\bf Skew braces of size $p^2q^2$. \\ The cyclic Sylow $q$-subgroup case}
		
\vspace{0.4cm}
\Large{Teresa Crespo}

\vspace{0.2cm}
\normalsize{Universitat de Barcelona, Gran Via de les Corts Catalanes 585, 08007, Barcelona (Spain), e-mail:teresa.crespo@ub.edu}

\end{center}

\begin{abstract}
We consider relatively prime integer numbers $m$ and $n$  such that each group of order $mn$ has a normal subgroup of order $m$. We prove that any skew brace of size $mn$ has an ideal of size $m$ and, if the subgroup of order $m$ of its additive group is abelian, a subbrace of size $n$. We give a method to classify skew braces of size $mn$, such that the subgroup of order $m$ of their additive group is abelian, from the classification of braces of size $m$ and the classification of skew braces of size $n$. We apply this result to determine all skew braces of size $p^2q^2$ of non-abelian type, for $p$ and $q$ odd primes, $p<q$, such that their ideal of order $q^2$ is cyclic.


\end{abstract}

\section{Introduction}

In \cite{R} Rump introduced an algebraic structure called brace to study
set-theoretic solutions of the Yang-Baxter equation. Guarnieri and Vendramin in \cite{GV} gave the more general concept of skew brace. A skew (left) brace is a set $B$ with two operations $\cdot$ and $\circ$ such
that $(B, \cdot)$ and $(B, \circ)$ are groups and the brace relation is satisfied, namely,
$$a\circ(b\cdot c) = a\circ b\cdot a^{-1} \cdot a\circ c,$$
for all $a, b, c \in B$, where $a^{-1}$ denotes the inverse of $a$ with respect to $\cdot$. We call $(B, \cdot$) the additive group and
$(B,\circ)$ the multiplicative group of the skew left brace. We say that $B$ is a brace (or a skew brace of abelian type) if its additive group is abelian. If $(B, \cdot)$ is a group, $(B, \cdot, \cdot)$ is a brace, called trivial brace.

Let $B_1$ and $B_2$ be skew left braces. A map $f : B_1 \to B_2$ is said to be a skew brace morphism if $f(a\cdot b) = f(a)\cdot f(b)$ and $f(a\circ b) = f(a)\circ f(b)$ for all $a,b \in B_1$. If $f$ is bijective, we say that $f$ is an isomorphism. In that case we say that the skew braces $B_1$ and $B_2$ are isomorphic.

In \cite{GV} Guarnieri and Vendramin proved that given a group $N$, there is
a bijective correspondence between isomorphism classes of skew left braces with additive group $N$ and conjugacy classes of regular subgroups of the holomorph $\Hol(N)$ of $N$.

We note that, if $N$ is a group and $G$ a regular subgroup of $\Hol (N)\simeq N \rtimes \Aut N$, then
${\pi_1}_{|G}:G\rightarrow N$
is bijective.

We recall that, for a skew brace $(B,\cdot,\circ)$ and each $a \in B$, we have a bijective map
$\lambda_a: B \rightarrow B$ defined by $\lambda_a(b)=a^{-1}\cdot  a\circ b$ which satisfies $\lambda_a(b\cdot c)=\lambda_a(b)\cdot \lambda_a(c), a\circ b=a \cdot\lambda_a(b),  \lambda_{a\circ b}=\lambda_a \lambda_b$, for any $a,b,c$ in $B$.

We state now precisely the correspondence between skew braces $(B,\cdot,\circ)$ and regular subgroups of the holomorph of $(B,\cdot)$.

\begin{proposition}[\cite{GV} th. 4.2]\label{GV} Let $(B,\cdot,\circ)$ be a skew left brace. Then

$$\{ (a,\lambda_a) \, : \, a \in B \}$$

\noindent
is a regular subgroup of $\Hol(B,\cdot)$, isomorphic to $(B,\circ)$.

Conversely, if $(B,\cdot)$ is a group and $G$ is a regular subgroup of $\Hol(B,\cdot)$, then $B$ is a skew left brace with $(B,\circ) \simeq G$, where

$$a \circ b=a\cdot f(b),$$

\noindent
and $({\pi_1}_{|G})^{-1}(a)=(a,f) \in G$.

These assignments give a bijective correspondence between the set of isomorphism classes of skew left braces $(B,\cdot,\circ)$ and the set of conjugation classes of regular subgroups of $\Hol(B,\cdot)$.
\end{proposition}

Skew left braces have been classified for sizes $p^2, p^3$, for $p$ a prime number (\cite{Ba2, NPhD, NZ}); $pq$ and $p^2q$, for $p$ and $q$ odd prime numbers (\cite{AB, AB2, AB3, CCD, D}); $2p^2$, for $p$ an odd prime number (\cite{C}); $8p$, for $p$ an odd prime number $\neq 3, 7$ in the abelian type case, (\cite{CGRV1}); for $12p$, for $p$ an odd prime number $\geq 7$ (\cite{CGRV2, CGRV3}) and for size $p^2q^2$, for $p$ and $q$ odd prime numbers in the abelian type case (\cite{C2} and \cite{C3}).

In this paper we consider relatively prime integer numbers $m$ and $n$  such that each group of order $mn$ has a normal subgroup of order $m$.  We prove that each skew brace of size $mn$ such that its additive group has an abelian subgroup of order $m$ may be obtained from a brace of size $m$ and a skew brace of size $n$. We further give a method to classify such skew braces of size $mn$ from the classification of braces of sizes $m$ and skew braces of size $n$ (cf. Theorem \ref{prop7}).  This is a generalization of the result obtained in \cite{CGRV2} in the case in which $m$ is prime and the one obtained in \cite{C2} for braces of size $mn$. We apply our result to describe all skew braces of size  $p^2q^2$, not of abelian type, for $p$ and $q$ odd primes, such that the subgroup of order $q^2$ of their additive group is cyclic.

\section{Skew left braces of size $mn$, for $\gcd(m,n)=1$}\label{method}

In the sequel we adopt the following hypotheses on the integer numbers $m$ and $n$.

\begin{equation}\label{hyp}
\begin{array}{r} \gcd(m,n)=1 \textit{\ and each group of order\ } mn \\ \textit{\ has a normal subgroup of order\ } m.
\end{array}
\end{equation}

Under these hypotheses the Schur-Zassenhaus theorem implies that every group $N$ of order $mn$ is a semidirect product of a normal subgroup $N_1$ of order $m$ and a subgroup $N_2$ of order $n$. Moreover $N_1$ is the unique subgroup of $N$ of order $m$ and every subgroup of $N$ of order $n$ is conjugate to $N_2$.

Let $N=N_1 \rtimes_{\sigma} N_2$, with $N_1$ a group of order $m$, $N_2$ a group of order $n$, with $\gcd(m,n)=1$. If $\varphi$ is an automorphism of $N$, then $\varphi(N_1)=N_1$ and $\varphi(N_2)$ is a subgroup of $N$ conjugate to $N_2$. Let

$$S= \{ \varphi \in \Aut N \, : \, \varphi(N_2)=N_2 \}.$$

\noindent
Clearly $S$ is a subgroup of $\Aut(N)$.
We give now a characterization of $S$ in terms of $\Aut N_1, \Aut N_2$ and the morphism $\sigma:N_2 \rightarrow \Aut N_1$ defining the semidirect product $N_1\rtimes_{\sigma} N_2$.

\begin{proposition}\label{prop} The image of the injective map

$$S \rightarrow \Aut N_1 \times \Aut N_2, \quad \varphi \mapsto (\varphi_{|N_1},\varphi_{|N_2})$$

\noindent
is precisely the set of pairs $(\varphi_1,\varphi_2) \in \Aut N_1 \times \Aut N_2$ such that

\begin{equation}\label{mor}\sigma \varphi_2=\conj_{\varphi_1} \sigma,
\end{equation}

\noindent
where $\conj_{\varphi_1}$ denotes conjugation by $\varphi_1$ in $\Aut N_1$.

\end{proposition}

\begin{proof} Given $\varphi_1 \in \Aut N_1, \varphi_2 \in \Aut N_2$, we want to determine when the bijective map $(\varphi_1,\varphi_2):N \rightarrow N$ defined by $(\varphi_1,\varphi_2)(x_1,x_2)=(\varphi_1(x_1),\varphi_2(x_2))$, for $x_1 \in N_1, x_2 \in N_2$, is a group morphism. For $x_1, y_1 \in N_1, x_2, y_2 \in N_2$, we have

$$\begin{array}{lll}(\varphi_1,\varphi_2)(x_1,x_2)(\varphi_1,\varphi_2)(y_1,y_2)&=&(\varphi_1(x_1),\varphi_2(x_2))(\varphi_1(y_1),\varphi_2(y_2))
\\&=& (\varphi_1(x_1)\sigma(\varphi_2(x_2))(\varphi_1(y_1)),\varphi_2(x_2)\varphi_2(y_2)).\end{array}$$

\noindent
On the other hand
$$\begin{array}{lll}(\varphi_1,\varphi_2)((x_1,x_2)(y_1,y_2))&=&(\varphi_1,\varphi_2)(x_1\sigma(x_2)(y_1),x_2y_2)\\
&=&(\varphi_1(x_1\sigma(x_2)(y_1)),\varphi_2(x_2y_2)).\end{array}$$

\noindent
Taking into account that $\varphi_1$ and $\varphi_2$ are morphisms, we obtain that $(\varphi_1,\varphi_2)$ is a morphism if and only if

$$\begin{array}{l} \sigma(\varphi_2(x_2))(\varphi_1(y_1))=\varphi_1(\sigma(x_2)(y_1)), \forall y_1 \in N_1, x_2 \in N_2
\\ \Leftrightarrow \sigma(\varphi_2(x_2))=\conj_{\varphi_1}(\sigma(x_2)), \forall x_2 \in N_2 \\
\Leftrightarrow \sigma\varphi_2=\conj_{\varphi_1}\sigma. \end{array}$$

\end{proof}

Let now $\varphi$ be any automorphism of $N$. Since $\varphi(N_2)$ is conjugate to $N_2$, there exists $\alpha \in N_1$ such that $\varphi(N_2)=\alpha N_2 \alpha^{-1}$, hence $\varphi=\conj_{\alpha} (\varphi_1,\varphi_2)$, for some $\varphi_1 \in \Aut N_1, \varphi_2 \in \Aut N_2$ as in Proposition \ref{prop} and where $\alpha$ is determined modulo $\Norm_N(N_2) \cap N_1$. We shall denote the automorphism $\varphi=\conj_{\alpha} (\varphi_1,\varphi_2)$ as a triple $(\alpha,\varphi_1,\varphi_2)$. With this notation,
for $a \in N_1, b \in N_2$, we have

\begin{equation}\label{im}
(\alpha,\varphi_1,\varphi_2)(a,b)=
(\alpha\varphi_1(a)\sigma(\varphi_2(b))(\alpha^{-1}),\varphi_2(b)).
\end{equation}

\noindent
Indeed, $(\alpha,\varphi_1,\varphi_2)(a,b)=\conj_{\alpha} (\varphi_1(a),\varphi_2(b))=(\alpha,e_2)(\varphi_1(a),\varphi_2(b))(\alpha,e_2)^{-1}=(\alpha\varphi_1(a),\varphi_2(b))(\alpha^{-1},e_2)=
(\alpha\varphi_1(a)\sigma(\varphi_2(b))(\alpha^{-1}),\varphi_2(b)),$ where $e_2$ denotes the identity element in $N_2$.

Now, using that $(\varphi_1,\varphi_2)$ satisfies \eqref{mor}, we obtain that the product in $\Aut N$ is given by

\begin{equation}\label{prod}
(\alpha,\varphi_1,\varphi_2)(\beta,\psi_1,\psi_2)=(\alpha\varphi_1(\beta),\varphi_1\psi_1,\varphi_2\psi_2).
\end{equation}

\begin{proposition}\label{ideal}
Let $m$ and $n$ be relatively prime integer numbers such that each group of order $mn$ has a normal subgroup of order $m$. Then each skew brace $B$ of size $mn$ has an ideal $B_1$ of size $m$, hence a quotient skew brace of size $n$. Moreover, for every $x \in B$ and every $y \in B_1$, we have $\lambda_y(x) x^{-1} \in B_1$.
\end{proposition}

\begin{proof} Let $(B,\cdot,\circ)$ be a skew brace of size $mn$. By the hypotheses its additive group $(B,\cdot)$ has a characteristic subgroup $B_1$ of order $m$. We have then, for $x, y \in B_1$, $x\circ y=x\lambda_x(y) \in B_1$ and for the inverse $\overline{x}$ of $x$ with respect to $\circ$, $\overline{x}=\lambda_x^{-1}(x^{-1}) \in B_1$. So $B_1$ is also a multiplicative subgroup of $B$ and, by the hypotheses, it is normal in $B$. Since $\lambda_x(B_1)=B_1$ for all $x \in B$, we obtain that $B_1$ is an ideal of $B$ of size $m$ and $B/B_1$ is a skew brace of size $n$.

Let now $x \in B, y \in B_1$. Since $B_1$ is a normal multiplicative subgroup of $B$, we have $x\circ y\circ \overline{x} \in B_1$. We compute this expression.

$$x\circ y\circ \overline{x}=(x \lambda_x(y))\circ \lambda_x^{-1}(x^{-1})=x\lambda_x(y)\lambda_{x \lambda_x(y)}(\lambda_x^{-1}(x^{-1})).$$

\noindent
Since $\lambda_{x \lambda_x(y)}=\lambda_{x \circ y}=\lambda_x\lambda_y$ and $\lambda_x^{-1}=\lambda_{\overline{x}}$, we obtain

$$x\circ y\circ \overline{x}=x\lambda_x(y)\lambda_{x\circ y\circ \overline{x}}(x^{-1}).$$

\noindent
The normality of $B_1$ as a multiplicative subgroup of $B$ is then equivalent to $x\lambda_x(y)\lambda_{x\circ y\circ \overline{x}}(x^{-1}) \in B_1$ for all $x \in B, y \in B_1$. Taking into account that $\lambda_x(y) \in B_1$ and that $B_1$ is a normal additive subgroup of $B$, this condition implies $\lambda_{x\circ y\circ \overline{x}}(x^{-1})x \in B_1$, for all $x \in B, y \in B_1$ and, in turn, $\lambda_y(x)x^{-1} \in B_1$, for all $x \in B, y \in B_1$.
\end{proof}

Given a brace $(B,\cdot,\circ)$, by Proposition \ref{GV},

$$G:=\{ (x,\lambda_x) \, : \, x \in B \}$$

\noindent
is a regular subgroup of $\Hol(B,\cdot)$, isomorphic to $(B,\circ)$. We consider now a brace of size $mn$, with $m$ and $n$ satisfying \eqref{hyp}. For the additive group $(B,\cdot)$ we have $B=B_1\rtimes_{\sigma} B_2$, for $B_1$ a normal additive subgroup of order $m$, $B_2$ an additive subgroup of order $n$ and $\sigma:(B_2,\cdot) \to \Aut (B_1,\cdot)$ a group morphism.

We shall denote the elements in $\Hol(B,\cdot)$ either by $(x,\varphi)$, with $x \in B, \varphi \in \Aut (B,\cdot)$ or, more explicitly, by $(a,b,\alpha,f,g)$, with $x=(a,b) \in B, \varphi=(\alpha,f,g) \in \Aut (B,\cdot)$. We denote by $e_{i}$ the identity element of $B_{i}$, $i=1,2$. The product in $\Hol(B,\cdot)$ is given by $(x,\varphi)(x',\varphi')=(x\varphi(x'),\varphi\varphi')$ and $(x,\varphi)^{-1}=(\varphi^{-1}(x^{-1}),\varphi^{-1})$, which implies

\begin{equation}\label{prodhol}
\begin{array}{l}(a,b,\alpha,f,g)(a',b',\alpha',f',g')\\ \quad =
(a\sigma(b)(\alpha f(a')\sigma(g(b'))(\alpha^{-1})),bg(b'),\alpha f(\alpha'),ff',gg');
\end{array}
\end{equation}

\begin{equation}\label{invhol}
\begin{array}{l}
(a,b,\alpha,f,g)^{-1} \\ \quad =
(f^{-1}(\alpha^{-1}\sigma(b^{-1})(a^{-1}\alpha)),g^{-1}(b^{-1}),f^{-1}(\alpha^{-1}),f^{-1},g^{-1}).
\end{array}
\end{equation}

\noindent
For $(x,\lambda_x) \in G$, we shall write $\lambda_x=(\alpha_x,f_x,g_x)$.

Let $N=N_1\rtimes_{\sigma} N_2$ such that $N_1$ is an abelian group. It is known that the set of automorphisms $\varphi$ of $N$ satisfying $\varphi(x)=x$ for all $x \in N_1$ and $\varphi(x)x^{-1} \in N_1$, for all $x \in N_2$, is in bijection with the set of 1-cocycles from $N_2$ to $N_1$ with the action given by $\sigma$. In the case when $\varphi(N_2)$ is conjugate to $N_2$, then the 1-cocycle corresponding to $\varphi$ is a 1-coboundary.

In the sequel we shall consider skew braces of size $mn$, with $m$ and $n$ satisfying hypotheses \eqref{hyp} with the additional condition that the normal subgroup of order $m$ of the additive group of the skew brace is an abelian group.

\begin{proposition}\label{B2}
Let $m$ and $n$ be relatively prime integer numbers such that each group of order $mn$ has a normal subgroup of order $m$. Let $(B,\cdot,\circ)$ be a skew brace of size $mn$ such that the normal subgroup of order $m$ of $(B,\cdot)$ is an abelian group. Then $B$ has a subbrace of size $n$.
\end{proposition}

\begin{proof} We consider the regular subgroup $G:=\{ (x,\lambda_x) \, : \, x \in B \}$
of $\Hol(B,\cdot)$, corresponding to $B$. By Proposition \ref{ideal}, $G_1:=\{ (x,\lambda_x) \, : \, x \in B_1 \}$ is a normal subgroup of $G$ and $G_1=\{(a,e_2,\alpha_a,f_a,\Id_{B_2}) \, : \, a \in B_1 \}$.  By the hypotheses \eqref{hyp}, $G$ is the semidirect product of $G_1$ and a subgroup $G_2$ of order $n$. Since each element of $G$ is equal to the product of an element in $G_1$ and an element in $G_2$, we can write $G_2$ in the following form

$$G_2=\{ (a_b,b,\alpha_b,f_b,g_b) \, :\, b \in B_2 \}.$$

\noindent
Using \eqref{prodhol} and \eqref{invhol}, we obtain that the subgroup condition for $G_2$ implies the following equalities, for $b,b' \in B_2$.

\begin{align}
\alpha_{bg_{b}(b')} &=\alpha_bf_{b}(\alpha_{b'}) \label{cs1} \\
 a_{bg_{b}(b')} &= a_{b} \sigma(b)(\alpha_b f_{b} (a_{b'}))\sigma(bg_{b}(b'))(\alpha_b^{-1})\label{cs2}\\
 f_{bg_{b}(b')} &= f_bf_{b'} \label{cs3}
\end{align}

We define

$$G_2^0:=\{ (e_1,b,e_1,f_b,g_b) \, :\, b \in B_2 \}.$$

\noindent
By \eqref{prodhol} and \eqref{invhol} and using that $G_2$ is a subgroup of $G$, it is clear that $G_2^0$ is also a subgroup of $G$. We want to see that, up to conjugation, we may assume $G=G_1\rtimes G_2^0$.

First observe that \eqref{cs1} gives that $b \mapsto \alpha_b$ is a 1-cocycle from $B_2$, with the product given by $(b,b') \mapsto bg_{b}(b')$, to $(B_1,\cdot)$ with $B_2$ acting on $(B_1,\cdot)$ by $\alpha^b=f_b(\alpha)$. Hence it is a 1-coboundary, i.e. there exists $z \in B_1$ such that $\alpha_b=zf_b (z^{-1} )$. Let $Z:=(e,\conj_z) \in \Hol(B,\cdot)$. We obtain

$$G_2^1:=ZG_2^0Z^{-1}=\{ (z\sigma(b)(z^{-1}),b,\alpha_b,f_b,g_b) \, :\, b \in B_2 \}.$$

\noindent
Now, $b\mapsto a_b$ and $b \mapsto z\sigma(b)(z^{-1})$ both satisfy \eqref{cs2}, hence $b\mapsto a_b \sigma(b)(z)z^{-1}$ is a 1-cocycle from $B_2$, with the product above, to $(B_1,\cdot)$ with $B_2$ acting on $B_1$ by $a^b=\sigma(b)(f_b(a))$, then a 1-coboundary. Let $c \in B_1$ such that $a_b \sigma(b)(z)z^{-1}=c\sigma(b)(f_b (c^{-1}))$ and $C:=(c,\Id_B) \in \Hol(B,\cdot)$. We obtain $CG_2^1C^{-1}=G_2$. Hence, we may assume $G=G_1\rtimes G_2^0$, as wanted. We note that the conjugate of $G_1$ under CZ may also be written as $\{ (a,e_2,\alpha_a,f_a,\Id) \, : \, a \in B_1\}$.

For $b,b' \in B_2$, we have $b\circ b'=b\lambda_b(b')=bg_b(b') \in B_2$, hence $B_2$ is a subbrace of $B$.

\end{proof}

\begin{remark} With the notations in the proof of Proposition \ref{B2}, we consider now

$$\overline{G_2}=\{ (b,g_b) \, :\, b \in B_2 \} \subset \Hol(B_2,\cdot).$$

\noindent
Clearly, $\overline{G_2}$ is a regular subgroup of $\Hol(B_2,\cdot)$ and, using \eqref{prodhol}, we see that $\overline{G_2}$ is isomorphic to $G_2^0$. The brace corresponding to $\overline{G_2}$ is then isomorphic to $B_2$.

\end{remark}

We specify now the action of $G_2^0$ on $G_1$ corresponding to the semidirect product structure.

\begin{equation}\label{action}
\begin{array}{l}
(e_1,b,e_1,f_b,g_b)(a,e_2,{\alpha}_a,f_a,\Id_{N_2})(e_1,g_b^{-1}(b^{-1}),e_1,f_b^{-1},g_b^{-1})
\\ = (\sigma(b)(f_b(a)),b,f_b({\alpha}_a),f_bf_a,g_b)(e_1,g_b^{-1}(b^{-1}),e_1,f_b^{-1},g_b^{-1})
\\ = (\sigma(b)(f_b(a))\sigma(b)(f_b(\alpha_a)\sigma(b^{-1})(f_b(\alpha_a^{-1}))),e_2,f_b({\alpha}_a),f_bf_af_b^{-1},\Id_{N_2})
\\ = (\sigma(b)(f_b(a\alpha_a))f_b(\alpha_a^{-1}),e_2,f_b({\alpha}_a),f_bf_af_b^{-1},\Id_{N_2}) \end{array}
\end{equation}

\begin{lemma}
We use the notations in the proof of Proposition \ref{B2}. For $b \in B_2$, we define $\tau(b):=\sigma(b) f_b$. Then

\begin{enumerate}[1)]
\item $\tau$ is a group morphism from $(B_2,\circ)$ to $\Aut(B_1,\cdot)$;
\item if $\alpha_a=e_1$, for all $a \in B_1$, $\tau(b) \in \Aut(B_1,\cdot,\circ)$.
\end{enumerate}
\end{lemma}

\begin{proof} 1) First we observe that $\tau(b)$ belongs to $\Aut(B_1,\cdot)$, since $\sigma(b)$ and $f_b$ do.
We note that, writing the product of two elements in $G_2$, we obtain that, for $b, b' \in B_2$, we have $f_b f_{b'}=f_{bg_b(b')}=f_{b \circ b'}$. Now

$$\tau (b\circ b')=\sigma(b\circ b') f_{b\circ b'}=\sigma(b)\sigma(g_b(b')) f_{b} f_{b'}=\sigma(b)f_b\sigma(b')f_{b'}=\tau(b)\tau(b'),
$$

\noindent
where, for the third equality, we have applied \eqref{mor} to $(f_b,g_b) \in \Aut (B,\cdot)$.

\noindent
2) If $\alpha_a=e_1$, for all $a \in B_1$, then \eqref{action} gives $f_{\tau(b)(a)}=f_bf_af_b^{-1}$. Since $(f_a,\Id) \in \Aut (B,\cdot)$, for all $a \in B_1$, \eqref{mor} gives that $f_a$ commutes with $\sigma(b)$ for all $b \in B_2$. We have then $f_{\tau(b)(a)}=\tau(b) f_a \tau(b)^{-1}$ which implies $\tau(b) \in \Aut(B_1,\circ)$.

\end{proof}

In the next theorem we prove that a skew brace $(B,\cdot,\circ)$ of size $mn$ is obtained from a skew brace of size $m$ and a skew brace of size $n$ under the hypotheses \eqref{hyp} and the assumption that the normal subgroup of order $m$ of $(B,\cdot)$ is abelian.

\begin{theorem}\label{prop7}
Let $N=N_1\rtimes_{\sigma} N_2$, with $|N_1|=m$, $N_1$ abelian, $|N_2|=n$ and $m, n$ satisfying hypotheses \eqref{hyp}, where $\sigma:N_2 \to \Aut N_1$ is a group morphism. Let $B_1$ be a skew brace with additive group isomorphic to $N_1$, $B_2$ a skew brace with additive group isomorphic to $N_2$  and let $G_1:= \{(a,\lambda_a)\, : \, a\in N_1 \}, G_2:=\{ (b,\lambda_b)\, : \, b \in N_2 \}$ be the corresponding regular subgroups of $\Hol(N_1)$ and $\Hol(N_2)$, respectively. Let $\tau:(B_2,\circ) \to \Aut(B_1,\cdot)$ be a group morphism.  Let us further assume

\begin{align}
\sigma(b) \lambda_a &= \lambda_a \sigma (b), \, \forall a \in N_1, b \in N_2, \label{eq9} \\
\sigma(b \lambda_b(c) b^{-1}) &= \tau(b)\sigma(c) \tau(b)^{-1}, \, \forall b,c \in N_2 \label{eq10}.
\end{align}

Then

$$\widetilde{G}_1:= \{ (a,e_2,\alpha_a,\lambda_a,\Id_{N_2}) \, : \, a \in N_1\}, \, \widetilde{G}_2:= \{ (e_1,b,e_1,\sigma(b)^{-1}\tau(b),\lambda_b) \, : \, b \in N_2\}$$

\noindent
are subgroups of $\Hol(N)$ such that $\widetilde{G}_2$ normalizes $\widetilde{G}_1$ and $G:=\widetilde{G}_1\rtimes \widetilde{G}_2$ is a regular subgroup of $\Hol(N)$, for each map $\alpha: N_1 \to N_1, a \mapsto \alpha_a$ satisfying

\begin{equation}\label{alpha0}
\alpha_{a\circ a'}=\alpha_a\lambda_a(\alpha_{a'}), \, \forall a,a' \in N_1;
\end{equation}

\vspace{0.3cm}
\noindent
and, for $A_b:=\tau(b)(a\alpha_a)\sigma(b)^{-1}\tau(b)(\alpha_a^{-1})$,

\begin{equation}\label{alpha1}
\alpha_{A_b} = \sigma(b)^{-1}\tau(b)(\alpha_a), \, \forall (a,b) \in N;
\end{equation}

\begin{equation}\label{alpha2}
\lambda_{A_b} = \sigma(b)^{-1}\tau(b) \lambda_a \tau(b)^{-1}\sigma(b), \, \forall (a,b) \in N.
\end{equation}

\vspace{0.3cm}
Moreover any skew brace $B$ of order $mn$ such that the normal subgroup of $(B,\cdot)$ of order $m$ is abelian is obtained in this way.
\end{theorem}

\begin{proof}

\noindent
1) Let us prove that $\widetilde{G}_1$ is a subgroup of $\Hol(N)$. Let us note first that \eqref{eq9} gives that $(\lambda_a,\Id) \in \Aut N$.

For $(a,e_2,\alpha_a,\lambda_a,\Id), (a',e_2,\alpha_{a'},\lambda_{a'},\Id) \in \widetilde{G}_1$, we have

    $$\begin{array}{lll}(a,e_2,\alpha_a,\lambda_a,\Id)(a',e_2,\alpha_{a'},\lambda_{a'},\Id)&=&
    (a\alpha_a\lambda_a(a')\alpha_a^{-1},e_2,\alpha_a\lambda_a(\alpha_{a'}),
    \lambda_a\lambda_{a'},\Id)\\ &=&(a\circ a',e_2,\alpha_{a\circ a'},\lambda_{a\circ a'},\Id),\end{array}$$

\noindent
using that $N_1$ is abelian, $\lambda \in \Hom((B_1,\circ),\Aut(B_1,\cdot))$ and $\alpha$ satisfies \eqref{alpha0}. Analogously

$$(a,e_2,\alpha_a,\lambda_a,\Id)^{-1}=(\lambda_a^{-1}(a^{-1}),e_2,\lambda_a^{-1}(\alpha_a^{-1}),\lambda_a^{-1},\Id)=
(\overline{a},e_2,\alpha_{\overline{a}},\lambda_{\overline{a}},\Id).$$

\noindent
2) Let us prove that $\widetilde{G}_2$ is a subgroup of $\Hol(N)$. We check first that $(\sigma(b)^{-1}\tau(b),\lambda_b) \in \Aut(N)$, i.e. that the pair $(\sigma(b)^{-1}\tau(b),\lambda_b)$ fulfils condition \eqref{mor} by using \eqref{eq10} and the fact that $\sigma$ is a group morphism from $N_2$ to $\Aut N_1$.

For $(e_1,b,e_1,\sigma(b)^{-1}\tau(b),\lambda_b), (e_1,b',e_1,\sigma(b')^{-1}\tau(b'),\lambda_{b'}) \in \widetilde{G}_2$, we have

$$\begin{array}{l} (e_1,b,e_1,\sigma(b)^{-1}\tau(b),\lambda_b)(e_1,b',e_1,\sigma(b')^{-1}\tau(b'),\lambda_{b'}) \\=
(e_1,b\lambda_b(b'),e_1,\sigma(b)^{-1}\tau(b)\sigma(b')^{-1}\tau(b'),\lambda_b\lambda_{b'})\\ =
(e_1,b\circ b',e_1,\sigma(b\circ b')^{-1}\tau(b\circ b'),\lambda_{b\circ b'}).\end{array}$$

\noindent
Indeed, since $\tau$ is a group morphism from $(B_2,\circ)$ to $\Aut(B_1,\cdot)$, it remains to prove  $\sigma(b)^{-1}\tau(b)\sigma(b')^{-1}=\sigma(b\circ b')^{-1} \tau(b)$ which follows from \eqref{eq10}. Analogously

$$\begin{array}{lll}(e_1,b,e_1,\sigma(b)^{-1}\tau(b),\lambda_b)^{-1}&=&(e_1,\lambda_b^{-1}(b^{-1}),e_1,\tau(b)^{-1}\sigma(b),\lambda_b^{-1})
\\ &=& (e_1,\overline{b},e_1,\sigma(\overline{b})^{-1}\tau(\overline{b}),\lambda_{\overline{b}}),\end{array}$$

\noindent
where the equality $\tau(b)^{-1}\sigma(b)=\sigma(\overline{b})^{-1}\tau(\overline{b})$ follows from \eqref{eq10}.

\noindent
3) Let us prove that $\widetilde{G}_2$ normalizes $\widetilde{G}_1$.

For $(e_1,b,e_1,\sigma(b)^{-1}\tau(b),\lambda_b) \in \widetilde{G}_2, (a,e_2,\alpha_a,\lambda_a,\Id) \in \widetilde{G}_1$, we have

$$\begin{array}{l}(e_1,b,e_1,\sigma(b)^{-1}\tau(b),\lambda_b) (a,e_2,\alpha_a,\lambda_a,\Id) (e_1,b,e_1,\sigma(b)^{-1}\tau(b),\lambda_b)^{-1}
\\=(e_1,b,e_1,\sigma(b)^{-1}\tau(b),\lambda_b) (a,e_2,\alpha_a,\lambda_a,\Id) (e_1,\lambda_b^{-1}(b^{-1}),e_1,\tau(b)^{-1}\sigma(b),\lambda_b^{-1})
\\=(\tau(b)(a\alpha_a)\sigma(b)^{-1}\tau(b)(\alpha_a^{-1}),e_2,\sigma(b)^{-1}\tau(b)(\alpha_a),\sigma(b^{-1})\tau(b) \lambda_a \tau(b)^{-1}\sigma(b),\Id)
\\=(A_b,e_2,\alpha_{A_b},\lambda_{A_b},\Id),
\end{array}
$$

\noindent
by using \eqref{alpha1} and \eqref{alpha2}.

\noindent
4) Let us prove that $G$ is a regular subgroup of $\Hol(N)$. Since

$$\begin{array}{l}(e_1,b,e_1,\sigma(b)^{-1}\tau(b),\lambda_b) (a,e_2,\alpha_a,\lambda_a,\Id)\\=(\tau(b)(a),b,\sigma(b)^{-1}\tau(b)(\alpha_a),\sigma(b)^{-1}\tau(b)\lambda_a,\lambda_b)
\end{array},$$

\noindent
we have $G=\{(\tau(b)(a),b,\sigma(b)^{-1}\tau(b)(\alpha_a),\sigma(b)^{-1}\tau(b)\lambda_a,\lambda_b):(a,b) \in N\}$. To see that $G$ is regular, it is enough to prove $\{(\tau(b)(a),b) \, : \, (a,b) \in N \}=N$, which is clear since, for $(a,b) \in N$, we have $(a,b)=(\tau(b)(\tau(b)^{-1}(a)),b)$.

\noindent
 5) Finally let us see that the conditions imposed to obtain the regular subgroup $G$ of $\Hol(N)$ are satisfied for any skew brace $B$ of order $mn$ such that the normal subgroup of $(B,\cdot)$ of order $m$ is abelian. With the notations in Proposition \ref{B2}, the fact that the pairs $(f_a,\Id_{B_2})$ and $(f_b,g_b)$ satisfy \eqref{mor} implies conditions \eqref{eq9} and \eqref{eq10}. Assuming $N_1$ abelian, the formula for the product of two elements in $G_1$, gives that $\alpha$ satisfies \eqref{alpha0}. Taking into account \eqref{action}, we obtain conditions \eqref{alpha1} and \eqref{alpha2}.
\end{proof}

\begin{remark}\label{rem}
With the hypotheses of Theorem \ref{prop7}, the trivial endomorphism $\alpha$ defined by $\alpha_a=e_1$ for all $a \in N_1$ satisfies \eqref{alpha0} and \eqref{alpha1}. Taking into account \eqref{eq9}, \eqref{alpha2} reduces to $\lambda_{\tau(b)(a)}=\tau(b) \lambda_a \tau(b)^{-1}, \, \forall (a,b) \in N$, which is fulfilled if $\tau(b) \in \Aut(B_1,\cdot,\circ)$. In this case, the brace corresponding to the regular subgroup $G$ of $\Hol(N)$ is the twofold semidirect product of
$B_1$ and $B_2$ via $\sigma$ and $\tau$ as defined in \cite{CGRV3} Definition 2.2, since \eqref{eq9} and \eqref{eq10} imply condition (2) in Proposition 2.1 of \cite{CGRV3}.
\end{remark}

\begin{examples}\label{exa}

Note that, if $B_1$ is a trivial brace, then \eqref{alpha0} gives that $\alpha$ is an endomorphism of $B_1$.

1) Assume $B_1$ is the trivial brace $\Z/(p)$. In this case, $\alpha$ is either trivial or an automorphism. In the second case, writing $\Z/(p)$ additively and identifying $\Aut \Z/(p)$ with $(\Z/(p))^*$, we obtain $A_b=\tau(b)a+\tau(b)(1-\sigma(b)^{-1}) \alpha_a$. Next \eqref{alpha1} implies $A_b=\sigma(b)^{-1}\tau(b)a$. Assuming $\sigma$ nontrivial, we obtain $\alpha_a=-a$. Since $B_1$ is trivial, we have $\lambda_a=\Id_{B_1}$ for all $a \in B_1$, hence \eqref{alpha2} is also satisfied for any $\alpha$. We recover the result obtained in \cite{CGRV3}. We note that for the results obtained there in Section 2 the hypothesis $p\nmid \theta(n)$ is not necessary.

2) Assume $B_1$ is the cyclic trivial brace of size $p^r$. Again, since $B_1$ is trivial, \eqref{alpha2} is  satisfied for any $\alpha$. Assume $\sigma$ nontrivial. If $\alpha$ is an automorphism, we obtain $\alpha_a=-a$, proceeding as in 1). If $\alpha$ is neither trivial nor an automorphism, we obtain that $\tau(b)(1-\sigma(b)^{-1}) (a+\alpha_a)$ is equal to a multiple of $p$. Since $p$ does not divide the order of $B_2$, we obtain that $1-\sigma(b)^{-1}$ cannot be a multiple of $p$ for all $b$, hence $p$ divides $a+\alpha_a$, for all $a$. In particular, $p$ should divide $1+\alpha_1$ but $p$ divides $\alpha_1$, hence there is no $\alpha$ with image of order $p^s, 1\leq s \leq r-1,$ satisfying \eqref{alpha1}.

\end{examples}

We want to see now when the regular subgroups $G, G'$ of $\Hol(N)$ corresponding to two quadruples $(B_1,B_2,\tau,\alpha)$ and $(B_1',B_2',\tau',\alpha')$ as in Theorem \ref{prop7} are conjugate. In \cite{BNY} Lemma 2.1, it is proved that $\Aut(N)$, as a subgroup of $\Hol(N)$, is action-closed with respect to the conjugation action of $\Hol(N)$ on
the set of regular subgroups of $\Hol(N)$. Hence it is enough to consider conjugation by elements of $\Aut N$. For

\begin{align*} (X,\Psi)&:=(a,b,\sigma(b)^{-1}\tau(b)(\alpha_{\tau(b)^{-1}(a)}),\sigma(b)^{-1}\tau(b)\lambda_{\tau(b)^{-1}(a)},\lambda_b) \in G,\\
 \Phi&:=(\beta,\varphi,\psi) \in \Aut N,\end{align*}

\noindent
 we have $\Phi(X,\Psi)\Phi^{-1}=(\Phi(X),\Phi \Psi\Phi^{-1})$. By calculation, we obtain

$$\Phi(X)=(\beta\varphi(a)\sigma(\psi(b))(\beta^{-1}),\psi(b)),$$

$$\Phi \Psi\Phi^{-1}=(\beta\varphi\sigma(b)^{-1}\tau(b)(\alpha_{\tau(b)^{-1}(a)}\lambda_{\tau(b)^{-1}(a)}\varphi^{-1}(\beta^{-1})),
\varphi\sigma(b)^{-1}\tau(b)\lambda_{\tau(b)^{-1}(a)}\varphi^{-1},\psi\lambda_b\psi^{-1}).$$

\noindent
We deduce $\psi\lambda_b\psi^{-1}=\lambda_{\psi(b)}$, equivalently $\psi \in \Aut(B_2,\cdot,\circ)$, and $(\psi(b),\lambda_{\psi(b)})$ is conjugate of $(b,\lambda_b)$ by $\psi$ in $\Hol(N_2)$. For $X \in \widetilde{G}_1$, i.e. $b=e_2, \lambda_b=\Id$, we obtain

$$\Phi(X)=(\varphi(a),e_2),\Phi \Psi\Phi^{-1}=(\beta\varphi(\alpha_a\lambda_a\varphi^{-1}(\beta^{-1})),\varphi\lambda_a\varphi^{-1},\Id).$$

\noindent
Since $\widetilde{G}_1$ is a characteristic subgroup of $G$, we deduce $\varphi\lambda_a\varphi^{-1}=\lambda_{\varphi(a)}$, equivalently $\varphi \in \Aut(B_1,\cdot,\circ)$ and $(\varphi(a),\lambda_{\varphi(a)})$ is conjugate of $(a,\lambda_a)$ by $\varphi$ in $\Hol(N_1)$. Moreover

\begin{equation}\label{alphas}
\alpha_{\varphi(a)}'=\beta\varphi(\alpha_a)\lambda_{\varphi(a)}(\beta^{-1}).
\end{equation}

Now

\begin{align*}\varphi\sigma(b)^{-1}\tau(b)\lambda_{\tau(b)^{-1}(a)}\varphi^{-1}&=
\varphi\sigma(b)^{-1}\tau(b)\varphi^{-1}\lambda_{\varphi(\tau(b)^{-1}(a))}\\
&=\sigma(\psi(b)^{-1})\varphi\tau(b)\varphi^{-1}\lambda_{\varphi(\tau(b)^{-1}(a))}
\end{align*}

\noindent
using $\varphi \in \Aut(B_1,\circ)$ and that $(\varphi,\psi)$ satisfies \eqref{mor}. We obtain then

\begin{equation}\label{taus}
\varphi\tau(b)\varphi^{-1}=\tau'(\psi(b))
\end{equation}

\noindent
and, finally,

$$\varphi\sigma(b)^{-1}\tau(b)\lambda_{\tau(b)^{-1}(a)}\varphi^{-1}=\sigma(\psi(b)^{-1})\tau'(\psi(b))\lambda_{\tau'(\psi(b))^{-1}(\varphi (a))}.$$

We have obtained the following result.

\begin{proposition}\label{class}
Two quadruples $(B_1,B_2,\tau,\alpha)$ and $(B_1',B_2',\tau',\alpha')$ as in Theorem \ref{prop7} provide isomorphic braces of size $mn$ if and only if there exist $(\beta,\varphi,\psi) \in \Aut N$ such that $\varphi$ is an isomorphism from $B_1$ to $B_1'$, $\psi$ is an isomorphism from $B_2$ to $B_2'$ and \eqref{alphas} and \eqref{taus} are satisfied.
\end{proposition}

\section{Braces of size $p^2$, for $p$ an odd prime number}\label{Bach}

In \cite{Ba2} Bachiller obtained the classification of braces of sizes $p^2$ and $p^3$, up to isomorphism, for $p$ a prime number. We recall it for braces $(B,+,\circ)$ of size $p^2$, for $p$ odd. We note that in this case $(B,\circ)$ is isomorphic to $(B,+)$. For each brace, we give the group of brace automorphisms and the subgroup $\Lambda$ of $\Aut(B,+)$ defined by $\Lambda=\{\varphi \in \Aut(B,+) : \lambda_a\varphi=\varphi\lambda_a , \forall a \in B \}$.

\subsection{$(B,+) \simeq \Z/(p^2)$}\label{cyclic}

There are two braces, up to isomorphism, with additive group isomorphic to $\Z/(p^2)$, the trivial one and a brace with $\circ$ defined by

$$a \circ b = a+ b +pab.$$

\noindent
Since $\Aut (\Z/(p^2))\simeq (\Z/(p^2))^*$ is abelian, in both cases, we have $\Lambda=\Aut (\Z/(p^2))$. In the trivial case, we have $\Aut B= \Aut (\Z/(p^2))$. In the nontrivial case, we have

$$\Aut B= \{ k \in (\Z/(p^2))^* \, : \, k \equiv 1 \pmod{p} \}.$$

We note that, whereas in the trivial case we have a unique brace structure, in the nontrivial one, there are $p-1$ braces $B_k$,  $k \in \{1,\dots, p-1 \}$ (isomorphic to each other), whose multiplicative law $\circ_k$ is obtained by transport of structure, i.e. $a\circ_k b=k^{-1}(ka\circ kb)$, which gives

$$a \circ_k b = a+ b +kpab.$$

\noindent
We have $\Aut B_k=\Aut B$, for all $k \in \{1,\dots, p-1 \}$. An isomorphism from $(B_k,\circ_k)$ into $\Z/(p^2)$ is given by $a \mapsto ka-pka(ka-1)/2$.

%In the nontrivial case, using the isomorphism from $(B,\circ)$ into $\Z/(p^2)$ given by $n\mapsto n-pn(n-1)/2$, we obtain %$\End(B,\circ)=\{\alpha_k \, : \, k \in \Z/(p^2)\}$ with $\alpha_k$ defined by $\alpha_k(n)=kn+pk(k-1)n^2/2$.

\subsection{$(B,+) \simeq \Z/(p)\times \Z/(p)$}\label{elem}

We write the elements in $\Z/(p)\times \Z/(p)$ in vector form. There are two braces, up to isomorphism, with additive group isomorphic to $\Z/(p)\times \Z/(p)$, the trivial one and a brace with $\circ$ defined by

$$\left( \begin{matrix} a_1\\ a_2 \end{matrix} \right) \circ \left(\begin{matrix} b_1 \\ b_2 \end{matrix} \right)=\left(\begin{matrix} a_1+b_1+a_2b_2\\ a_2+b_2\end{matrix} \right) = \left( \begin{matrix} a_1\\ a_2 \end{matrix} \right) + \left(\begin{matrix} b_1 \\ b_2 \end{matrix} \right) + a_2b_2 \left( \begin{matrix} 1 \\ 0 \end{matrix} \right) .$$

\noindent
In the trivial case, we have $\Lambda=\Aut B= \Aut (\Z/(p)\times \Z/(p))\simeq \GL(2,p)$. In the nontrivial case, we have

$$\Aut B= \left\{ \left( \begin{array}{cc} d^2 & b \\ 0 & d \end{array} \right) \, : \, b \in \Z/(p), d \in (\Z/(p))^*\right\},$$

$$\Lambda= \left\{ \left( \begin{array}{cc} d & b \\ 0 & d \end{array} \right) \, : \, b \in \Z/(p), d \in (\Z/(p))^*\right\}.$$

We note that, whereas in the trivial case we have a unique brace structure, in the nontrivial one, we obtain braces $B_M$ (isomorphic to each other), by transport of structure, i.e. $a\circ_M b=M^{-1}(Ma\circ Mb)$,  for $M \in \GL(2,p)$. More precisely, for $M=\left(\begin{smallmatrix} m_{11} & m_{12} \\ m_{21} & m_{22}\end{smallmatrix} \right)$ running over a transversal $T$ of $\Aut B$ in $\GL(2,p)$, we obtain $p^2-1$ different braces whose multiplicative law $\circ_M$ is given by

$$\left( \begin{matrix} a_1\\ a_2 \end{matrix} \right) \circ_M \left(\begin{matrix} b_1 \\ b_2 \end{matrix} \right)=\left(\begin{matrix} a_1+b_1\\ a_2+b_2\end{matrix} \right)+\dfrac 1 {\det M} \left(\begin{matrix} m_{22}(m_{21} a_1+m_{22} a_2)(m_{21} b_1+m_{22} b_2)\\  -m_{21}(m_{21} a_1+m_{22} a_2)(m_{21} b_1+m_{22} b_2) \end{matrix} \right).$$

\noindent
We have $\Aut B_M=M^{-1} (\Aut B) M$, for all $M \in T$. Writing $v:=\left(\begin{smallmatrix} 1 \\ 0 \end{smallmatrix}\right)$, we obtain

$$a\circ_M b=a+b+\theta(a)\theta(b)\, w,$$

\noindent
where $\theta$ denotes the linear form given by the second row of $M$, i.e. $\theta(a)=m_{21}a_1+m_{22}a_2$, and $w:=\frac 1{\det M}\left(\begin{smallmatrix} m_{22}\\ -m_{21}\end{smallmatrix}\right)=M^{-1}v$. We have $\theta(w)=(MM^{-1}v)_2=0$, i.e. $w$ is a nonzero vector of $\Ker \theta$.

\section{Skew left braces of size $p^2q^2$}

Let $p$ and $q$ be odd primes  with $p<q$. By the Sylow's theorems, a group of order $p^2q^2$ has a normal subgroup of order $q^2$ and a subgroup of order $p^2$. In \cite{C2} and \cite{C3} we classified skew braces of abelian type of size $p^2q^2$. In this section we determine skew braces of size $p^2q^2$ of non-abelian type such that the subgroup of order $q^2$ of their additive group is cyclic. To obtain a non trivial action of the order $p^2$ subgroup on the order $q^2$ subgroup, $p$ and $q$ must satisfy $p \mid q-1$.
The numbers $m=q^2$ and $n=p^2$ satisfy the conditions in Theorem \ref{prop7}. Hence, every skew brace of size $p^2q^2$ may be obtained from a brace $B_1$ of size $q^2$, a brace $B_2$ of size $p^2$, group morphisms $\sigma:(B_2,\cdot) \to \Aut(B_1,\cdot)$ and $\tau:(B_2,\circ) \to \Aut(B_1,\cdot)$  and a map $\alpha$ from $B_1$ to $B_1$ satisfying the conditions in Theorem \ref{prop7}. We use the results on braces of order $p^2$ given in Section \ref{Bach} and Theorem \ref{prop7} to determine all skew braces of non-abelian type of size $p^2q^2$, such that the subgroup of order $q^2$ of their additive group is cyclic.

In this section, we denote by $+$ the additive law of $B_1$ and $B_2$.

\subsection{$(B_1,+)=\Z/(q^2)$ and $(B_2,+)=\Z/(p^2)$}

In this case, if $p^2 \nmid q-1$, $\sigma:\Z/(p^2) \to (\Z/(q^2))^*, 1 \mapsto u$, for $u$ a fixed element of order $p$ in $(\Z/(q^2))^*$. If $p^2 \mid q-1$, we can also have $\sigma':\Z/(p^2) \to (\Z/(q^2))^*, 1 \mapsto u'$, for $u'$ a fixed element of order $p^2$ in $(\Z/(q^2))^*$. Since $\Aut B_1$ is abelian, \eqref{mor} reduces to $\sigma \psi=\sigma$ and \eqref{eq10} reduces to $\sigma(b \lambda_b(c) b^{-1}) = \sigma(c)$ and analogously for $\sigma'$. We determine $(\varphi,\psi)$ satisfying \eqref{mor}. We obtain that $\varphi$ is any element in $\Aut B_1=\Aut N_1$. We have then one equivalence class for $B_1$ under the relation defined in Proposition \ref{class}. When the action is given by $\sigma$, $\psi \in \{ k \in (\Z/(p^2))^* \, : \, k\equiv 1 \pmod{p}\}$ and, when it is given by $\sigma'$, $\psi = 1_{\Z/(p^2)}$.

\subsubsection{$B_1$ trivial brace}\label{B1t}

 From Examples \ref{exa} we know that $\alpha$ satisfying \eqref{alpha0}, \eqref{alpha1} and \eqref{alpha2} is given by either $\alpha_a=0$, for all $a$, or $\alpha_a=-a$.

\vspace{0.2cm}
\noindent
{\bf If $B_2$ is trivial,} we have one brace structure for $B_2$. Now \eqref{eq10} is satisfied for any $\sigma$ or $\sigma'$ as above and any $\tau:(B_2,\circ) \to \Aut (B_1,\cdot)$.

We assume first $p^2 \nmid q-1$. We may write $\tau(1)=u^{\ell}$, for some $\ell, 0 \leq \ell \leq p-1$. Since $\psi \in \{ k \in (\Z/(p^2))^* \, : \, k\equiv 1 \pmod{p}\}$, every $\tau$ stands alone in its equivalence class under \eqref{taus}. We obtain $2p$ braces.

We assume now $p^2 \mid q-1$ and that the action of $(B_2,\cdot)$ on $(B_1,\cdot)$ is given by $\sigma$. We may write $\tau_{\ell}(1)=u'^{\ell}$, for some $\ell, 0 \leq \ell \leq p^2-1$. Since $\psi \in \{ k \in (\Z/(p^2))^* \, : \, k\equiv 1 \pmod{p}\}$, $\tau_{\ell}$, with $p\mid \ell$  stands alone in its equivalence class under \eqref{taus} while the $\tau_{\ell}$, with $p\nmid \ell$, are grouped in $p-1$ classes. We have then $2p-1$ classes of $\tau$'s. We obtain $2(2p-1)$ braces.

Finally we assume $p^2 \mid q-1$ and that the action of $(B_2,\cdot)$ on $(B_1,\cdot)$ is given by $\sigma'$. We may write $\tau(1)=u'^{\ell}$, for some $\ell, 0 \leq \ell \leq p^2-1$. Since $\psi=1_{\Z/(p^2)}$, every $\tau$ stands alone in its  class under \eqref{taus}.  We obtain $2p^2$ braces.

\vspace{0.2cm}
\noindent
{\bf If $B_2$ is nontrivial,} $\Aut B_2 = \{ k \in (\Z/(p^2))^* \, : \, k\equiv 1 \pmod{p}\}$. We have then $p-1$ classes for $B_2$ under the relation given in Proposition \ref{class}. Now $\lambda_b(c)=c+pbc$ and \eqref{eq10} gives $\sigma(c+pbc) = \sigma(c)$, which is fulfilled, whereas $\sigma'(c+pbc) = \sigma'(c)$ is not.

We assume first $p^2 \nmid q-1$. Again we may write $\tau(1)=u^{\ell}$, for some $\ell, 0 \leq \ell \leq p-1$. As above, since, in \eqref{taus}, $\psi \in \Aut(B_2,\cdot,\circ)=\{ k \in (\Z/(p^2))^* \, : \, k \equiv 1 \pmod{p} \}$, every $\tau$ stands alone in its conjugation class. We obtain then $2p(p-1)$ braces.

We assume now $p^2 \mid q-1$. Since $\sigma'$ does not satisfy \eqref{eq10}, the action of $(B_2,\cdot)$ on $(B_1,\cdot)$ is given by $\sigma$. As above we may write $\tau(1)=u'^{\ell}$, for some $\ell, 0 \leq \ell \leq p^2-1$  and, since $\psi \in \{ k \in (\Z/(p^2))^* \, : \, k\equiv 1 \pmod{p}\}$, we have $2p-1$ classes of $\tau$'s under \eqref{taus}. We obtain $2(2p-1)(p-1)$ braces.

\subsubsection{$B_1$ nontrivial brace}\label{4.1.1}

Since $B_1$ is nontrivial, $\lambda_a=1+qa$, for $a \in B_1$.

\vspace{0.2cm}
\noindent
{\bf If $B_2$ is trivial,} we have one brace structure for $B_2$. We determine now $\alpha:B_1 \to B_1$ satisfying \eqref{alpha0}, \eqref{alpha1} and \eqref{alpha2}. Note that, applying successively \eqref{alpha0}, we obtain $\alpha(1\circ \overset{k}{\cdots} \circ 1)=(k+qk(k-1)/2)\alpha_1$ and, for $a \in B_1$, we have $a=1\circ \overset{\tilde{a}}{\cdots} \circ 1$, where $\tilde{a}=a-qa(a-1)/2$. Hence \eqref{alpha0} implies $\alpha_a=a \alpha_1$, where the product is the ring one in $\Z/(q^2)$, i.e. $\alpha \in \End(B_1,+)$. Next we obtain that \eqref{alpha1} is satisfied for $\alpha_1=0$ and $\alpha_1=-1$, both for $\sigma$ and for $\sigma'$. Now, if $\alpha_1=0$, \eqref{alpha2} is satisfied for $\tau(b)=1$, for all $b \in N_2$ and, if $\alpha_1=-1$, \eqref{alpha2} is satisfied for $\tau=\sigma$ or $\tau=\sigma'$, respectively. We obtain then two pairs $(\alpha,\tau)$ and then $2$ braces both for $\sigma$ and for $\sigma'$.

\vspace{0.2cm}
\noindent
{\bf If $B_2$ is nontrivial,} since $\psi \in \{ k \in (\Z/(p^2))^* \, : \, k\equiv 1 \pmod{p}\}$, we have $p-1$ brace structure classes for $B_2$. We have $\lambda_b(c)=c+pbc$. As in Section \ref{B1t}, \eqref{eq10} is fulfilled for $\sigma$ and not for $\sigma'$. As above, we obtain two possible pairs $(\alpha,\tau)$. We obtain then $2(p-1)$ braces when the action of $(B_2,\cdot)$ on $(B_1,\cdot)$ is given by $\sigma$ and 0 when it is given by $\sigma'$.

\vspace{0.3cm}
We have obtained the following result.

\begin{proposition}
Let $p$ and $q$ be odd prime numbers such that $p\mid q-1$. Let $u$ denote a fixed element of order $p$ in $(\Z/(q^2))^*$. If $p^2\mid q-1$, let $u'$ denote a fixed element of order $p^2$ in $(\Z/(q^2))^*$.

\noindent
1) On $\Z/(q^2)\times \Z/(p^2)$ we define the operation

$$(a,b)\cdot (a',b') = (a+u^b a',b+b'), \, \text{\ for \ }a,a' \in \Z/(q^2), b,b' \in \Z/(p^2).$$

\noindent
a) We assume $p^2\nmid q-1$. Then there are $2p^2+2p$ skew braces, up to isomorphism, with additive group $(\Z/(q^2)\times \Z/(p^2),\cdot)$ whose multiplicative laws are given by

$$
\begin{array}{l}
(a,b) \circ (a',b')= (a+u^{\ell b} a',b+b'), \, 0\leq \ell\leq p-1; \\
(a,b) \circ (a',b')= (a+u^{\ell kb} a',b+b'+kpbb'), \, 0\leq \ell\leq p-1, 1\leq k \leq p-1; \\
(a,b) \circ (a',b')= (u^{b'}a+u^{\ell b}a',b+b'), \, 0\leq \ell\leq p-1;\\
(a,b) \circ (a',b')= (u^{b'}a+u^{\ell kb}a',b+b'+kpbb'), \, 0\leq \ell\leq p-1, 1\leq k \leq p-1;\\
(a,b) \circ (a',b')= (a+a'+qaa',b+b'+kpbb'), \, 0\leq k \leq p-1; \\
(a,b) \circ (a',b')= (u^{b'}a+u^ba'+qaa',b+b');\\
(a,b) \circ (a',b')= (u^{b'}a+u^{kb}a'+qaa',b+b'+kpbb'), \, 1\leq k \leq p-1.\\
\end{array}
$$

\noindent
b) We assume $p^2\mid q-1$. Then there are $4p^2$ skew braces, up to isomorphism, with additive group $(\Z/(q^2)\times \Z/(p^2),\cdot)$. These are the braces given in a) and besides those whose multiplicative laws are given by

$$
\begin{array}{l}
(a,b) \circ (a',b')= (a+u'^{\ell b} a',b+b'), \, 1\leq \ell\leq p-1; \\
(a,b) \circ (a',b')= (a+u'^{\ell (kb-pkb(kb-1)/2)} a',b+b'+kpbb'), \, 1\leq \ell\leq p-1, 1\leq k \leq p-1; \\
(a,b) \circ (a',b')= (u^{b'}a+u'^{\ell b}a',b+b'), \, 1\leq \ell\leq p-1;\\
(a,b) \circ (a',b')= (u^{b'}a+u'^{\ell (kb-pkb(kb-1)/2)}a',b+b'+kpbb'), \, 1\leq \ell\leq p-1, 1\leq k \leq p-1.
\end{array}
$$

\noindent
2) If $p^2\mid q-1$, on $\Z/(q^2)\times \Z/(p^2)$ we define the operation

$$(a,b)\cdot (a',b') = (a+u'^b a',b+b'), \, \text{\ for \ }a,a' \in \Z/(q^2), b,b' \in \Z/(p^2).$$

\noindent
Then there are $2p^2+2$ skew braces, up to isomorphism, with additive group $(\Z/(q^2)\times \Z/(p^2),\cdot)$ whose multiplicative laws are given by

$$
\begin{array}{l}
(a,b) \circ (a',b')= (a+u'^{\ell b} a',b+b'), \, 0\leq \ell\leq p^2-1; \\
(a,b) \circ (a',b')= (u'^{b'}a+u'^{\ell b}a',b+b'), \, 0\leq \ell\leq p^2-1;\\
(a,b) \circ (a',b')= (a+a'+qaa',b+b'); \\
(a,b) \circ (a',b')= (u'^{b'}a+u'^{b}a'+qaa',b+b').\\
\end{array}
$$

In the description of the multiplicative laws we have taken into account the isomorphism from $B_k$ into $\Z/(p^2)$ given in Section \ref{cyclic}, when appropriate.
\end{proposition}

\subsection{$(B_1,+)=\Z/(q^2)$ and $(B_2,+)=\Z/(p)\times \Z/(p)$}\label{4.2}

In this case $\sigma:\Z/(p)\times \Z/(p) \to (\Z/(q^2))^*$.  Note that, since $\Aut(B_1,+)$ is abelian, any $\sigma$ satisfies \eqref{eq9}; \eqref{mor} reduces to $\sigma \psi=\sigma$ and \eqref{eq10} reduces to $\sigma(b \lambda_b(c) b^{-1}) = \sigma(c)$. For this last condition to be met when $B_2$ is not trivial, we take $\sigma$ defined by $\left(\begin{smallmatrix} 1 \\ 0 \end{smallmatrix}\right) \mapsto 1, \left(\begin{smallmatrix} 0 \\ 1 \end{smallmatrix}\right) \mapsto u$,  for $u$ a fixed element of order $p$ in $(\Z/(q^2))^*$. We determine $(\varphi,\psi)$ satisfying \eqref{mor}. We obtain that $\varphi$ is any element in $\Aut N_1$ and $\psi \in \{ \left(\begin{smallmatrix} a&b\\c&d \end{smallmatrix} \right) \in \GL(2,p) \, : \, c=0,d=1\}$. We have one isomorphism class for $B_1$.

\subsubsection{$B_1$ trivial brace}\label{triv}

\vspace{0.2cm}
\noindent
{\bf If $B_2$ is trivial,} there is one isomorphism class for $B_2$. Next \eqref{eq10} is satisfied for any $\tau:(B_2,\circ) \to \Aut(B_1,\cdot)$. Such a $\tau$ is determined by  $\left(\begin{smallmatrix} 1 \\ 0 \end{smallmatrix}\right) \mapsto u^k, \left(\begin{smallmatrix} 0 \\ 1 \end{smallmatrix}\right) \mapsto u^{\ell}$, for some $k,\ell \in \{0,\cdots, p-1\}$. Since $(B_2,\circ)=(B_2,+)$, we have $\tau\left(\begin{smallmatrix} x_1 \\ x_2 \end{smallmatrix}\right)=u^{kx_1+\ell x_2}$. Modulo \eqref{taus} with $\psi=\left(\begin{smallmatrix} a&b\\0&1\end{smallmatrix} \right), a \in (\Z/p\Z)^*, c \in \Z/p\Z$, we obtain $p+1$ classes, which we describe by giving the corresponding pair $(k,\ell)$. These are $[(1,0)]=\{(k,\ell) \, : \, k\neq 0\},[(0,\ell)]=\{(0,\ell)\},0\leq \ell \leq p-1$. Since $B_1$ is trivial, \eqref{alpha0} is equivalent to $\alpha\in \End(B_1,+)$. We have then $\alpha_a=a\alpha_1$ and we obtain two possible values for $\alpha_1$, namely $\alpha_1=0$ and $\alpha_1=-1$ (see Example \ref{exa} 2)). We obtain then $2(p+1)$ braces.

\vspace{0.2cm}
\noindent
{\bf If $B_2$ is nontrivial,} $\Aut B_2 = \{ \left( \begin{smallmatrix} d^2 & b \\ 0 & d \end{smallmatrix} \right) \, : \, b \in \Z/(p), d \in (\Z/(p))^*\}$. We determine first the nontrivial brace structures on $N_2$ satisfying \eqref{eq10} and their behaviour under the automorphisms of $N_2$ allowed by \eqref{mor}.

\begin{lemma}\label{lemB2}
Let $N_2=\Z/(p)\times \Z/(p)$, let $v=\left(\begin{smallmatrix} 1\\ 0\end{smallmatrix}\right)$ and let $\sigma:\Z/(p)\times \Z/(p) \to (\Z/(q^2))^*$ be defined as at the beginning of Section \ref{4.2}, so that $\Ker \sigma=\langle v \rangle$. For $s \in (\Z/(p))^*$, let $\circ_s$ be the operation on $N_2$ defined by

$$\left( \begin{matrix} x_1\\ x_2 \end{matrix} \right) \circ_s \left(\begin{matrix} y_1 \\ y_2 \end{matrix} \right)=\left(\begin{matrix} x_1+y_1+s\, x_2y_2\\ x_2+y_2\end{matrix} \right).$$

\noindent
Then

\begin{enumerate}
\item[1)] the nontrivial brace structures on $N_2$ satisfying \eqref{eq10} are exactly the $p-1$ braces $(N_2,+,\circ_s), s \in (\Z/(p))^*$;
\item[2)] for $a \in (\Z/(p))^*, b \in \Z/(p)$, the automorphism $\psi=\left(\begin{smallmatrix} a&b\\0&1\end{smallmatrix} \right)$ of $(N_2,+)$ is an isomorphism of braces from $(N_2,+,\circ_s)$ onto $(N_2,+,\circ_{as})$.
\end{enumerate}

\noindent
In particular, the braces in 1) are permuted transitively by the automorphisms $\psi$ of $N_2$ allowed by \eqref{mor} and, for a fixed $s$, those preserving $\circ_s$ are exactly the ones with $a=1$.
\end{lemma}

\begin{proof}
1) For the multiplicative law $\circ_M$ defined in Section \ref{elem}, we have $\lambda_x(y)=-x+x\circ_M y=y+\theta(x)\theta(y)w$. Since $(B_2,\cdot)$ is abelian, \eqref{eq10} reads $\sigma(\lambda_x(y))=\sigma(y)$, for all $x,y \in N_2$, i.e. $\sigma(\theta(x)\theta(y)w)=1$, for all $x,y \in N_2$, i.e. $w \in \Ker \sigma$. Since $w\neq 0$ spans $\Ker \theta$ and $\Ker \sigma=\langle v\rangle$, we obtain $\Ker \theta=\langle v \rangle$, hence $\theta(x)=\delta x_2$ and $w=t\, v$, for some $\delta, t \in (\Z/(p))^*$, and then $\circ_M=\circ_s$, for $s=\delta^2t$. Conversely, for $s \in (\Z/(p))^*$, we have $\circ_s=\circ_M$ with $M=\left(\begin{smallmatrix} s^{-1}&0\\0&1\end{smallmatrix} \right)$, and $\circ_s$ satisfies \eqref{eq10}, since in this case $w=s\, v \in \Ker \sigma$. Clearly $\circ_s\neq \circ_{s'}$, for $s \neq s'$.

\noindent
2) We have $\psi^{-1}(x)=\left(\begin{smallmatrix} a^{-1}(x_1-bx_2)\\ x_2\end{smallmatrix}\right)$ and hence

\begin{align*}
\psi\left(\psi^{-1}(x) \circ_s \psi^{-1}(y)\right)&=\left(\begin{matrix} (x_1-bx_2)+(y_1-by_2)+as\, x_2y_2+b(x_2+y_2)\\ x_2+y_2\end{matrix} \right)\\
&=\left(\begin{matrix} x_1+y_1+as\, x_2y_2\\ x_2+y_2\end{matrix} \right)=x \circ_{as} y.
\end{align*}

\noindent
Since $\psi \in \Aut (N_2,+)$, this means precisely that $\psi$ is an isomorphism of braces from $(N_2,+,\circ_s)$ onto $(N_2,+,\circ_{as})$. The last assertions follow, since $a$ runs over $(\Z/(p))^*$ and $as=s$ if and only if $a=1$.
\end{proof}

\noindent
By Lemma \ref{lemB2}, there is one brace structure class for $B_2$ and we may take $B_2=(N_2,+,\circ_1)$; the automorphisms $\psi$ at our disposal in \eqref{taus} are then exactly those preserving $\circ_1$, i.e. those with $a=1$. In this case, $\tau:(B_2,\circ) \to \Aut(B_1,\cdot)$ is as well determined by  $\left(\begin{smallmatrix} 1 \\ 0 \end{smallmatrix}\right) \mapsto u^k, \left(\begin{smallmatrix} 0 \\ 1 \end{smallmatrix}\right) \mapsto u^{\ell}$, for some $k,\ell \in \{0,\cdots, p-1\}$ but in this case we have $\tau\left(\begin{smallmatrix} x_1 \\ 0 \end{smallmatrix}\right)=u^{kx_1}$ and, for $x_2\neq 0$, $\tau\left(\begin{smallmatrix} x_1 \\ x_2 \end{smallmatrix}\right)=u^{k(x_1-x_2(x_2-1)/2)+\ell x_2}$, since $\left(\begin{smallmatrix} x_1 \\ x_2 \end{smallmatrix}\right)=\left(\begin{smallmatrix} 1 \\ 0 \end{smallmatrix}\right) \circ \overset{z}{\dots} \circ \left(\begin{smallmatrix} 1 \\ 0 \end{smallmatrix}\right)\circ \left(\begin{smallmatrix} 0 \\ 1 \end{smallmatrix}\right)\circ \overset{x_2}{\dots} \circ \left(\begin{smallmatrix} 0 \\ 1 \end{smallmatrix}\right)$, with $z=x_1-x_2(x_2-1)/2$. Modulo \eqref{taus} with $\psi=\left(\begin{smallmatrix} 1&b\\0&1\end{smallmatrix} \right), b \in \Z/p\Z$, which gives $(k,\ell)\mapsto (k,\ell-kb)$, we obtain $2p-1$ classes, which we describe by giving the corresponding pair $(k,\ell)$. These are $[(k,0)]=\{(k,\ell) \, : \, 0\leq \ell \leq p-1\},1\leq k\leq p-1;[(0,\ell)]=\{(0,\ell)\},0\leq \ell \leq p-1$. As above, there are two possible solutions for $\alpha$. We obtain then $2(2p-1)$ braces.

\subsubsection{$B_1$ nontrivial brace}

Since $B_1$ is nontrivial, $\lambda_a=1+qa$, for $a \in B_1$.

\vspace{0.2cm}
\noindent
{\bf If $B_2$ is trivial,} as in Section \ref{triv}, there is one brace structure class for $B_2$. As in Section \ref{4.1.1}, we obtain that $\alpha_a=a\alpha_1$, for $a \in N_1$ and two possible pairs $(\alpha_1,\tau)$, namely $(0,1)$ and $(-1,\sigma)$. We have then 2 braces.

\vspace{0.2cm}
\noindent
{\bf If $B_2$ is nontrivial,} by Lemma \ref{lemB2}, there is one brace structure class for $B_2$. Again, there are  two possible pairs $(\alpha_1,\tau)$ and hence 2 braces.

\vspace{0.3cm}
We have obtained the following result.

\begin{proposition}
Let $p$ and $q$ be odd prime numbers such that $p\mid q-1$. Let $u$ denote a fixed element of order $p$ in $(\Z/(q^2))^*$.

\noindent
1) On $\Z/(q^2)\times (\Z/(p)\times \Z/(p))$ we define the operation

$$\left(a,\left( \begin{smallmatrix} b_1\\ b_2 \end{smallmatrix} \right)\right)\cdot \left(a',\left( \begin{smallmatrix} b_1'\\ b_2' \end{smallmatrix} \right)\right) = (a+u^{b_2} a',\left( \begin{smallmatrix} b_1+b_1'\\ b_2+b_2' \end{smallmatrix} \right)),$$

\noindent
for $a,a' \in \Z/(q^2), \left( \begin{smallmatrix} b_1\\ b_2 \end{smallmatrix} \right),\left( \begin{smallmatrix} b_1'\\ b_2'\end{smallmatrix} \right) \in \Z/(p)\times \Z/(p)$. There are $6p+4$ skew braces, up to isomorphism, with additive group $(\Z/(q^2)\times (\Z/(p)\times \Z/(p)),\cdot)$ whose multiplicative laws are given by

$$
\begin{array}{l}
\left(a,\left( \begin{smallmatrix} b_1\\ b_2 \end{smallmatrix} \right)\right)\circ \left(a',\left( \begin{smallmatrix} b_1'\\ b_2' \end{smallmatrix} \right)\right) = \left( a+u^{b_1}a', \left( \begin{smallmatrix} b_1+b_1'\\ b_2+b_2' \end{smallmatrix} \right)\right);\\[10pt]
\left(a,\left( \begin{smallmatrix} b_1\\ b_2 \end{smallmatrix} \right)\right)\circ \left(a',\left( \begin{smallmatrix} b_1'\\ b_2' \end{smallmatrix} \right)\right) = \left( u^{b_2'}a+u^{b_1}a', \left( \begin{smallmatrix} b_1+b_1'\\ b_2+b_2' \end{smallmatrix} \right)\right);\\[10pt]
\left(a,\left( \begin{smallmatrix} b_1\\ b_2 \end{smallmatrix} \right)\right)\circ \left(a',\left( \begin{smallmatrix} b_1'\\ b_2' \end{smallmatrix} \right)\right) =\left( a+u^{\ell b_2}a', \left( \begin{smallmatrix} b_1+b_1'\\ b_2+b_2' \end{smallmatrix} \right)\right), \, 0\leq \ell \leq p-1; \\[10pt]
\left(a,\left( \begin{smallmatrix} b_1\\ b_2 \end{smallmatrix} \right)\right)\circ \left(a',\left( \begin{smallmatrix} b_1'\\ b_2' \end{smallmatrix} \right)\right) = \left( u^{b_2'}a+u^{\ell b_2}a', \left( \begin{smallmatrix} b_1+b_1'\\ b_2+b_2' \end{smallmatrix} \right)\right), \, 0\leq \ell \leq p-1;\\[10pt]
\left(a,\left( \begin{smallmatrix} b_1\\ b_2 \end{smallmatrix} \right)\right)\circ \left(a',\left( \begin{smallmatrix} b_1'\\ b_2' \end{smallmatrix} \right)\right) = \left(a+u^{k(b_1-b_2(b_2-1)/2)}a', \left( \begin{smallmatrix} b_1+b_1'+b_2b_2'\\ b_2+b_2' \end{smallmatrix} \right)\right), \, 1\leq k \leq p-1;\\[10pt]
\left(a,\left( \begin{smallmatrix} b_1\\ b_2 \end{smallmatrix} \right)\right)\circ \left(a',\left( \begin{smallmatrix} b_1'\\ b_2'
\end{smallmatrix} \right)\right) = \left(u^{b_2'}a+u^{k(b_1-b_2(b_2-1)/2)}a', \left( \begin{smallmatrix} b_1+b_1'+b_2b_2'\\ b_2+b_2' \end{smallmatrix} \right)\right), \, 1\leq k \leq p-1;\\[10pt]
\left(a,\left( \begin{smallmatrix} b_1\\ b_2 \end{smallmatrix} \right)\right)\circ \left(a',\left( \begin{smallmatrix} b_1'\\ b_2' \end{smallmatrix} \right)\right) = \left(a+u^{\ell b_2}a', \left( \begin{smallmatrix} b_1+b_1'+b_2b_2'\\ b_2+b_2' \end{smallmatrix} \right)\right), \, 0\leq \ell \leq p-1;\\[10pt]
\left(a,\left( \begin{smallmatrix} b_1\\ b_2 \end{smallmatrix} \right)\right)\circ \left(a',\left( \begin{smallmatrix} b_1'\\ b_2'
\end{smallmatrix} \right)\right) = \left(u^{b_2'}a+u^{\ell b_2}a', \left( \begin{smallmatrix} b_1+b_1'+b_2b_2'\\ b_2+b_2' \end{smallmatrix} \right)\right), \, 0\leq \ell \leq p-1;\\[10pt]
\left(a,\left( \begin{smallmatrix} b_1\\ b_2 \end{smallmatrix} \right)\right)\circ \left(a',\left( \begin{smallmatrix} b_1'\\ b_2'
\end{smallmatrix} \right)\right) = \left(a+a'+qaa', \left( \begin{smallmatrix} b_1+b_1'\\ b_2+b_2' \end{smallmatrix} \right)\right);\\[10pt]
\left(a,\left( \begin{smallmatrix} b_1\\ b_2 \end{smallmatrix} \right)\right)\circ \left(a',\left( \begin{smallmatrix} b_1'\\ b_2'
\end{smallmatrix} \right)\right) = \left(u^{b_2'}a+u^{ b_2}a'+qaa', \left( \begin{smallmatrix} b_1+b_1'\\ b_2+b_2'\end{smallmatrix} \right)\right);\\[10pt]
\left(a,\left( \begin{smallmatrix} b_1\\ b_2 \end{smallmatrix} \right)\right)\circ \left(a',\left( \begin{smallmatrix} b_1'\\ b_2'
\end{smallmatrix} \right)\right) = \left(a+a'+qaa', \left( \begin{smallmatrix} b_1+b_1'+b_2b_2'\\ b_2+b_2' \end{smallmatrix} \right)\right);\\[10pt]
\left(a,\left( \begin{smallmatrix} b_1\\ b_2 \end{smallmatrix} \right)\right)\circ \left(a',\left( \begin{smallmatrix} b_1'\\ b_2'
\end{smallmatrix} \right)\right) = \left(u^{b_2'}a+u^{ b_2}a'+qaa', \left( \begin{smallmatrix} b_1+b_1'+b_2b_2'\\ b_2+b_2'\end{smallmatrix} \right)\right).\\
\end{array}
$$

\end{proposition}

\end{document}